\documentclass[11pt,twoside]{article}
\usepackage{a4}
\usepackage{amssymb,amsmath,amsthm,latexsym}
\usepackage{amsfonts}
\usepackage{amsfonts}
\usepackage{graphicx}
\usepackage{nicematrix}

\newtheorem{theorem}{Theorem}[section]

\newtheorem{corollary}[theorem] {Corollary}
\newtheorem{definition}[theorem]{Definition}
\newtheorem{example}[theorem]{Example}
\newtheorem{lemma} [theorem]{Lemma}

\newtheorem{proposition}[theorem]{Proposition}
\newtheorem{remark}[theorem]{Remark}

\usepackage{caption}
\usepackage{tikz}

\vspace{5cm}

\begin{document}

\label{'ubf'}  
\setcounter{page}{1}                                 

\markboth {\hspace*{-9mm} \centerline{\footnotesize \sc
	\scriptsize {A note on connectivity preserving splitting operation for matroids representable over $GF(p)$}}
}
{ \centerline                           {\footnotesize \sc  
	Prashant Malavadkar$^1,$ Sachin Gunjal$^2$ and Uday Jagadale$^3$                                                  } \hspace*{-9mm}              
}

\begin{center}
{ 
	{\Large \textbf { \sc 	A note on connectivity preserving splitting operation for matroids representable over $GF(p)$
		}
	}
	\\

	\medskip

	{\sc Prashant Malavadkar$^1,$ Sachin Gunjal$^2$ and Uday Jagadale$^3$  }\\
	{\footnotesize  School of Mathematics and Statistics, MIT World Peace University,
		Pune 411 038,	India. }\\
	{\footnotesize e-mail: {\it 1. prashant.malavadkar@mitwpu.edu.in, 2. sachin.gunjal@mitwpu.edu.in,\\ 3. uday.jagdale@mitwpu.edu.in   }}
}
\end{center}

\thispagestyle{empty}

\hrulefill

\begin{abstract}  
{\footnotesize 
\noindent The splitting operation on a $p$-matroid does not necessarily preserve connectivity. It is observed that there exists a single element extension of the splitting matroid which is connected. In this paper, we define the element splitting operation on a $p$-matroids which is a splitting operation followed by a single element extension. It is proved that the element splitting operation on connected $p$-matroid yields a connected $p$-matroid. We give a sufficient condition to yield Eulerian $p$-matroids from Eulerian $p$-matroids under the element splitting operation. A sufficient condition to obtain hamiltonian $p$-matroid by applying the element splitting operation on $p$-matroid is also provided.
	
}
\end{abstract}
\hrulefill



\section{Introduction}

\noindent We discuss loopless and coloopless $p$-matroids, by $p$-matroid we mean a vector matroid $M\cong M[A]$ for some matrix $A$ of size $m \times n$ over the field $F = GF(p),$ for prime $p$. We denote the set of column labels of $M$ (viz. the ground set of $M$) by $E$, the set of circuits of $M$ by $\mathcal {C}(M),$ and the set of independent sets of $M$ by $\mathcal {I}(M).$ For undefined, standard terminology in graphs and matroids, see Oxley \cite{ox}.\\

\noindent Malavadkar et al. \cite{mjg} defined the splitting operation for $p$-matroids as :
\begin{definition}\label{def0}
	Let $M\cong M[A]$ be a $p$-matroid on the ground set $E,$$\{a,b\} \subset E,$ and $\alpha\neq 0$ in $ GF(p)$. The matrix $A_{a,b}$ is constructed  from $A$ by appending an extra row to $A$ which has coordinates equal to $\alpha$ in the columns corresponding to the elements $a,$$b,$ and zero elsewhere. Define the splitting matroid $M_{a,b}$ to be the vector matroid  $M[A_{a,b}].$ The transformation of $M$ to $M_{a,b}$ is called the splitting operation.
\end{definition}
\noindent A circuit $C \in \mathcal{C}(M)$ containing $\{a,b\}$ is said to be a $p$-circuit of $M,$ if $C \in \mathcal{C}(M_{a,b}).$ And if $C$ is a circuit of $M$ containing either $a$ or $b,$ but it is not a circuit of $M_{a,b},$ then we say $C$ is an $np$-circuit of $M.$ For $a,b\in E,$ if the matroid $M$ contains no $np$-circuit then splitting operation on $M$ with respect to $a,b$ is called trivial splitting. \\ 
Note that the class of connected $p$-matroids is not closed under splitting operation. 
\begin{example}\label{ex2}
	The vector matroid $M \cong M[A]$ represented by the matrix $A$ over the field $GF(3)$ is connected, whereas the splitting matroid $M_{1,4}\cong M[A_{1,4}] $ is not connected.
	
	\begin{center}
		$\mathbf{A} = 
		\begin{pNiceMatrix}%
			[first-col,
			first-row,
			code-for-first-col = \color{black},
			code-for-first-row = \color{black}]
			& 1 & 2 & 3 & 4 & 5 & 6 & 7 & 8 \\
			& 1 & 0 & 0 & 0 & 0 & 1 & 1 & 2 \\
			& 0 & 1 & 0 & 0 & 1 & 0 & 1 & 1 \\
			& 0 & 0 & 1 & 0 & 1 & 1 & 0 & 1 \\
			& 0 & 0 & 0 & 1 & 2 & 1 & 1 & 0 \\
			& 0 & 1 & 1 & 0 & 0 & 0 & 0 & 0
		\end{pNiceMatrix} \qquad
		\mathbf{A_{1,4}} = 
		\begin{pNiceMatrix}%
			[first-col,
			first-row,
			code-for-first-col = \color{black},
			code-for-first-row = \color{black}]
			& 1 & 2 & 3 & 4 & 5 & 6 & 7 & 8 \\
			& 1 & 0 & 0 & 0 & 0 & 1 & 1 & 2 \\
			& 0 & 1 & 0 & 0 & 1 & 0 & 1 & 1 \\
			& 0 & 0 & 1 & 0 & 1 & 1 & 0 & 1 \\
			& 0 & 0 & 0 & 1 & 2 & 1 & 1 & 0 \\
			& 0 & 1 & 1 & 0 & 0 & 0 & 0 & 0 \\
			& 1 & 0 & 0 & 1 & 0 & 0 & 0 & 0	
		\end{pNiceMatrix}$
	\end{center}
	 
\end{example}

 It is interesting to see that the vector matroid $M'_{1,4} \cong M[A'_{1,4}],$ which is a single element extension of $M_{1,4},$ is connected.\\

\begin{center}
	$\mathbf{A'_{1,4}} = 
	\begin{pNiceMatrix}%
		[first-col,
		first-row,
		code-for-first-col = \color{black},
		code-for-first-row = \color{black}]
		& 1 & 2 & 3 & 4 & 5 & 6 & 7 & 8 & 9\\
		& 1 & 0 & 0 & 0 & 0 & 1 & 1 & 2 & 0\\
		& 0 & 1 & 0 & 0 & 1 & 0 & 1 & 1 & 0\\
		& 0 & 0 & 1 & 0 & 1 & 1 & 0 & 1 & 0\\
		& 0 & 0 & 0 & 1 & 2 & 1 & 1 & 0 & 0\\
		& 0 & 1 & 1 & 0 & 0 & 0 & 0 & 0 & 0\\
		& 1 & 0 & 0 & 1 & 0 & 0 & 0 & 0	& 1
	\end{pNiceMatrix}$
\end{center}

\noindent This example motivates us to investigate the question: If $M$ is a connected $p$-matroid and $M_{a,b}$ is the splitting matroid of $M,$ then does there exist a single element extension of the splitting matroid that is connected? In the next section, we answer this question by defining the element splitting operation on a $p$-matroid $M$ which is splitting operation on $M$ followed by a single element extension.

\section{Element Splitting Operation}
	
In this section, we define the element splitting operation on a $p$-matroid $M$ and characterize its circuits.   

\begin{definition}\label{def1}
	
Let $M\cong M[A]$ be a $p$-matroid on the ground set $E,$$\{a,b\} \subset E,$ and $M_{a,b}$ be the corresponding splitting matroid. Let the matrix $A_{a,b}$ represents $M_{a,b}$ on $GF(p).$ Construct the matrix $A'_{a,b}$ from $A_{a,b}$ by adding an extra column to $A_{a,b},$labeled as $z,$ which has the last coordinate equal to $\alpha\neq 0$ and the rest are equal to zero. Define the element splitting matroid $M'_{a,b}$ to be the vector matroid  $M[A'_{a,b}]$. The transformation of $M$ to $M'_{a,b}$ is called the element splitting operation.

\end{definition}
\noindent Splitting and element splitting operations on binary matroids are closely studied in \cite{mds, mal, Mills, rsw, ma, shika, ba}. A matroid $L$ is a lift of the matroid $M,$ if there exists a matroid $N,$ and $X\subset E(N)$ such that $N/X=M,$ and $N \setminus X=L.$ If $X$ is a singleton set, then $L$ is called an elementary lift of $M.$ In the following result, Mundhe et al. \cite{mundhe} showed the equivalence of splitting matroid with  elementary lift for binary matroids: 
	\begin{lemma}
		Let $M$ and $L$ be binary matroids. Then $L$ is an elementary lift of $M$ if and only if $L$ is isomorphic to $M_T$ for some $T\subset E(M).$ 
	\end{lemma}
	\noindent Lemma 2.2 can be extended to $p$-matroids  by using the similar arguments used to prove it   in \cite{mundhe}. Thus a splitting matroid $M_{a,b}$ of $p$-matroid $M$ is an elementary lift of  $M.$ In-depth study on lifted graphic matroid is done in \cite{CG, CW, FM}.

\begin{remark}
  rank$(A)<$ rank$(A'_{a,b})=$ rank$(A)+1.$ If the rank functions of $M$ and $M'_{a,b}$ are denoted by $r$ and $r',$respectively, then $r(M) < r'(M'_{a,b})= r(M) + 1.$
\end{remark}

\noindent Let $C=\{v_1,v_2,\dots,v_k\},$ where $v_i, i= 1,2, \ldots,k$ are column vectors of the matrix $A,$ be an $np$-circuit of $M$ containing only $a.$ Assume $v_1=a,$ without loss of generality.  Then there exist non-zero scalars $\alpha_1,\alpha_2,\dots,\alpha_k \in GF(p)$ such that $\alpha_1 v_1+ \alpha_2 v_2+ \ldots+\alpha_k v_k \equiv 0(\mod p).$ Let $\alpha_z \in GF(p)$ be such that $\alpha_z + \alpha_1  \equiv 0 (\mod p).$ Note that $\alpha_z \neq 0.$ Then in the matrix $A'_{a,b},$ we have $\alpha_1 v_1+ \alpha_2 v_2+ \ldots+\alpha_k v_k + \alpha_z z  \equiv 0(\mod p).$ Therefore the set $C \cup z=\{v_1,v_2,\dots,v_k, z\}$ is a dependent set of $M'_{a,b}.$ If both $a,b\in C,$ then by the similar arguments, we can show that $C \cup z$ is a dependent set of $M'_{a,b}.$ 
 
\noindent In the next Lemma, we characterize the circuits of $M'_{a,b}$ containing the element $z.$ 
\begin{lemma}\label{L1}
Let $C$ be a circuit of $p$-matroid $M.$ Then $C \cup z$ is a circuit of $M'_{a,b}$ if and only if $C$ is an $np$-circuit of $M.$
\end{lemma}
\begin{proof}
 
First assume that $C \cup z$ is a circuit of $M'_{a,b}.$ If $C$ is not an $np$-circuit of $M,$ then it is a $p$-circuit of $M,$ and hence it is also a circuit of $M_{a,b}$ and $M'_{a,b},$ as well. Thus we get a circuit $C$ contained in $C \cup z,$ a contradiction.	
	
\noindent Conversely, suppose $C$ is an $np$-circuit of $M.$ Then $C$ is an independent set of $M'_{a,b}.$ As noted earlier, $C \cup z$ is a dependent set of $M'_{a,b}.$ On the contrary, assume that $C \cup z$ is not a circuit of $M'_{a,b},$ and $C_1 \subset C \cup z $ be a circuit of $M'_{a,b}.$\\
\textbf{Case 1}: $z \notin C_1.$ Then $C_1$ is a circuit contained in $C,$ which is contradictory to the fact that $C$ is independent in $M'_{a,b}.$\\
\textbf{Case 2}: $z \in C_1.$ Then $C_1 \setminus z$ is a dependent set of $M$ contained in the circuit $C$ which is not possible. Thus we get $ C \cup z$ is a circuit of $M'_{a,b}.$  

\end{proof}
\noindent We denote the collection of circuits described in Lemma \ref{L1} by $\mathcal {C}_z.$

\begin{theorem}
Let $M$ be a $p$-matroid on the ground set $E$ and $\{a,b\}\subset E.$ Then $\mathcal{C}(M'_{a,b})= \mathcal{C}(M_{a,b})\cup\mathcal {C}_z.$
\end{theorem}
\begin{proof}
The inclusion $\mathcal{C}(M_{a,b})\cup\mathcal {C}_z \subset \mathcal{C}(M'_{a,b})$ follows from the Definition \ref{def1} and Lemma \ref{L1}. For the other inclusion, let $C \in \mathcal{C}(M'_{a,b}).$ If $z \notin C,$ then $C \in \mathcal{C}(M_{a,b}).$ Otherwise, $C \in \mathcal {C}_z.$

\end{proof}

\begin{example}\label{ex1}
	Consider the matroid $R_8,$ the vector matroid of the following matrix $A$ over field $GF(3)$.
	
	\begin{center}
		
		$\mathbf{A} =
		\begin{pNiceMatrix}%
			[first-col,
			first-row,
			code-for-first-col = \color{blue},
			code-for-first-row = \color{blue}]
			& 1 & 2 & 3 & 4 & 5 & 6 & 7 & 8   \\
			&  1 & 0 & 0 & 0 & 2 & 1 & 1 & 1  \\
			&  0 & 1 & 0 & 0 & 1 & 2 & 1 & 1  \\
			&  0 & 0 & 1 & 0 & 1 & 1 & 2 & 1  \\
			&  0 & 0 & 0 & 1 & 1 & 1 & 1 & 2  \\
		\end{pNiceMatrix} \qquad
		\mathbf{A'_{3,5}} =
		\begin{pNiceMatrix}%
			[first-col,
			first-row,
			code-for-first-col = \color{blue},
			code-for-first-row = \color{blue}]
			& 1 & 2 & 3 & 4 & 5 & 6 & 7 & 8 & 9 \\
			&  1 & 0 & 0 & 0 & 2 & 1 & 1 & 1 & 0 \\
			&  0 & 1 & 0 & 0 & 1 & 2 & 1 & 1 & 0\\
			&  0 & 0 & 1 & 0 & 1 & 1 & 2 & 1 & 0 \\
			&  0 & 0 & 0 & 1 & 1 & 1 & 1 & 2 & 0 \\
			&  0 & 0 & 1 & 0 & 1 & 0 & 0 & 0 & 1 \\
		\end{pNiceMatrix}$
	\end{center}
	
	\noindent For $a=3$, $b=5$ and $\alpha=1$ the representation of  element splitting matroid $M'_{3,5}$ over $GF(3)$ is given by the matrix $A'_{3,5}$.
	The collection of circuits of $M,$ $M_{3,5}$ and $M'_{3,5}$ is given in the following table.
	
	\begin{center}
		\begin{tabular}{ | m{4cm}| m{4cm} | m{5cm} |}
			\hline
			\textbf{~~~~~~~~~Circuits of $M$}  & \textbf{~~~~~~~~Circuits of $M_{3,5}$} & \textbf{~~~~~~~~Circuits of $M'_{3,5}$}\\
			\hline
			$\{1, 2, 3, 4,5\}$, $\{1, 2, 7, 8\}$, $\{1, 4, 6, 7\}$ & $\{1, 2, 3, 4,5\}$, $\{1, 2, 7, 8\}$, $\{1, 4, 6, 7\}$ & $\{1, 2, 3, 4,5\}$, $\{1, 2, 7, 8\}$, $\{1, 4, 6, 7\}$ \\
			
			\hline
			$ \{2, 4, 6, 8\}$,$\{3, 5, 6, 7, 8\}$ & $\{2, 4, 6, 8\}$,$\{3, 5, 6, 7, 8\}$  & $\{2, 4, 6, 8\}$,$\{3, 5, 6, 7, 8\}$ \\
			\hline

			- & $\{1, 2, 3, 5, 6, 7\}$, $\{1, 2, 3, 5, 6, 8\}$  & $\{1, 2, 3, 5, 6, 7\}$, $\{1, 2, 3, 5, 6, 8\}$\\
			\hline

			-& $\{1, 3, 4, 5, 6, 8\}$, $\{1, 3, 4, 5, 7, 8\}$  & $\{1, 3, 4, 5, 6, 8\}$, $\{1, 3, 4, 5, 7, 8\}$\\
			\hline

			-& $\{2, 3, 4, 5, 6, 7\}$, $\{2, 3, 4, 5, 7, 8\}$  & $\{2, 3, 4, 5, 6, 7\}$, $\{2, 3, 4, 5, 7, 8\}$ \\
			\hline
			
			$\{1, 2, 3, 4, 6\}$, $\{1, 2, 3, 4, 7\}$ & -  & $\{1, 2, 3, 4, 6, 9\}$, $\{1, 2, 3, 4, 7, 9\}$ \\
			\hline
			
			$ \{1, 2, 3, 4, 8\}$, $\{1, 2, 5, 6\} $ & -  & $\{1, 2, 3, 4, 8, 9\}$, $\{1, 2, 5, 6, 9\}$ \\
			\hline
			
			$\{1, 3, 5, 7\}$, $\{1, 3, 6, 8\}$ & - & $\{1, 3, 5, 7, 9\}$, $\{1, 3, 6, 8, 9\}$ \\
			\hline
			
			$\{1, 4, 5, 8\}$, $\{1, 5, 6, 7, 8\} $ & -  & $\{1, 4, 5, 8, 9\}$, $\{1, 5, 6, 7, 8, 9\} $ \\
			\hline
			
			$\{2, 3, 5, 8\}$, $\{2, 3, 6, 7\}$ & - & $\{2, 3, 5, 8, 9\}$, $\{2, 3, 6, 7, 9\}$ \\
			\hline
			
			$\{2, 4, 5, 7\}$, $\{2, 5, 6, 7, 8\}$ & - & $\{2, 4, 5, 7, 9\}$, $\{2, 5, 6, 7, 8, 9\}$ \\
			\hline
			
			$\{3, 4, 5, 6\}$, $\{3, 4, 7, 8\}$ & - & $\{3, 4, 5, 6, 9\}$, $\{3, 4, 7, 8, 9\}$ \\
			\hline
			
			$\{4, 5, 6, 7, 8\}$ & - & $\{4, 5, 6, 7, 8, 9\}$ \\
			\hline
		\end{tabular}
	\end{center}

\end{example}

\subsection{Independent sets, Bases and Rank function of $M'_{a,b}$}
In this section, we describe independent sets, bases and rank function of $M'_{a,b}.$ Denote the set $\mathcal{I}_z=\{I\cup z : I\in \mathcal{I}(M)\}.$  

\begin{lemma}\label{Indep}
Let $M\cong M[A]$ be a $p$-matroid with the ground set $E$ and $M'_{a,b}$ be its element splitting matroid. Then $\mathcal{I}(M'_{a,b})=\mathcal{I}(M_{a,b})\cup \mathcal{I}_z$ 
\end{lemma}
\begin{proof}
Notice that $\mathcal{I}(M_{a,b})\cup \mathcal{I}_z \subseteq \mathcal{I}(M'_{a,b}).$ For other inclusion, assume $ T \in \mathcal{I}(M'_{a,b}).$ If $z\notin T,$ then $T\in \mathcal{I}(M_{a,b}).$ And if $z\in T,$ then $T\setminus\{z\} \in \mathcal{I}(M_{a,b}).$ That is $T=I\cup z$ for some $I\in\mathcal{I}(M_{a,b}).$
\begin{description}
	\item \textbf{Case 1} : $I\in \mathcal{I}(M).$ Then $T\in \mathcal{I}_z.$
	\item \textbf{Case 2} : $I=C\cup I'$ where $C$ is an $np$-circuit of $M$ and $I' \in \mathcal {I}(M).$ Then by Lemma \ref{L1}, $C\cup z$ is a circuit of $M'_{a,b}$ contained in $T,$ a contradiction.
\end{description}
\end{proof}
\begin{lemma}\label{k}
Let $M$ be a $p$-matroid and $\{a,b\} \subset E.$ Then $\mathcal{B}(M'_{a,b})=\mathcal{B}(M_{a,b}) \cup \mathcal{B}_z,$ where $\mathcal{B}_z = \{ B \cup z : B \in \mathcal{B}(M)\}.$
\end{lemma}

\begin{proof}
It is easy to observe that $\mathcal{B}(M_{a,b})\cup \mathcal{B}_z \subseteq \mathcal{B}(M'_{a,b}).$ Next assume that $B \in \mathcal{B}(M'_{a,b}).$ Then $rank (B)=rank (M)+1.$ If $B$ contains $z,$ then $B\setminus z$ is an independent set of $M_{a,b}$ of size $rank(M).$ Then by similar arguments given in the proof of Lemma \ref{Indep}, $B = I \cup z$, for some $I \in \mathcal{I}(M)$. Therefore $B\setminus z$  is a basis of $M$ and $B \in \mathcal{B}_z.$ If $z\notin B,$ then $B$ is an independent set of size $rank(M)+1.$ Therefore $B \in \mathcal{B}(M_{a,b}).$
\end{proof}

\noindent In the following lemma, we provide the rank function of $M'_{a,b}$ in terms of the rank function of $M.$
\begin{lemma}\label{s}
	Let $r$ and $r'$ be the rank functions of the matroids $M$ and $M'_{a,b},$ respectively. Suppose $S\subseteq E(M).$ Then $r'(S \cup z) = r(S)+1, $ and
	\begin{equation}
		\begin{split}
			r'(S) &= r(S) , \text { ~~~~~  if S contains no np-circuit of M; and}\\
			&= r(S)+1,\text {~   if S contains an np-circuit of M.}
		\end{split}
	\end{equation}
\end{lemma}

\begin{proof}
The equality $r'(S \cup z) = r(S)+1$ follows from the definition. The proof of the Equation(1) is discussed in Corollary 2.13 of \cite{mjg}.
\end{proof}

\section{Connectivity of element splitting $p$-matroids}
Let $M$ be a matroid having the ground set $E,$ and $k$ be a positive integer. The $k$-separation of matroid $M$ is a partition $\{S, T\}$ of $E$ such that $|S|, |T|\geq k$ and $r(S) + r(T)- r(M) <k.$ For an integer $n \geq 2,$ we say $M$ is an $n$-connected if $M$ has no $k$- separation, where $1 \leq k \leq n-1.$ 

 \noindent In the following theorem, we provide a necessary and sufficient condition to preserve the connectedness of a $p$-matroid under element splitting operation.
\begin{theorem}\label{con1}
	Let $M$ be a connected $p$-matroid on the ground set $E.$ Then  $M'_{a,b}$  is a connected $p$-matroid on the ground set $E\cup \{z\}$ if and only if $M_{a,b}$ is the splitting matroid obtained by applying non-trivial splitting operation on $M.$ 
\end{theorem}

\begin{proof}
	First assume that $M'_{a,b}$  is a connected $p$-matroid on the ground set $E\cup \{z\}.$ On the contrary, suppose $M_{a,b}$ is obtained by applying trivial splitting operation. Then $M$ contains no $np$ circuits with respect to the splitting by elements $a,b.$ Now, let $S=\{z\}$ and $T=E.$ Then  $r'(S) + r'(T) -r'(M'_{a,b}) = 1 + r (E)- (r(M)+1) = 0 < 1$ gives a 1-separation of $M'_{a,b},$ which is a contradiction.

	\noindent For converse part, assume that $M_{a,b}$ is the splitting matroid obtained by applying non-trivial splitting operation on $M.$  Suppose that, $M'_{a,b}$ is not connected. It means $M'_{a,b}$ has $1$-separation, say $\{S,T\}.$ Then $|S|, |T|\geq 1$ and
	\begin{equation}\label{eq2}
		 r'(S) + r'(T) -r'(M'_{a,b}) < 1.
	\end{equation}  
	$\mathbf{Case~ 1:}$ Assume $S =\{z\}.$ Then $T$ contains an $np$ circuit. Then Equation \ref{eq2} gives, $1+(1+r(T))-r(M)-1 < 1 \implies$ $r(T)<r(M),$ which is not possible.\\ 
	$\mathbf{Case~ 2:}$ Assume $|S| \geq 2,$  $ z\in S.$ If $T$ contains no $np$-circuit then Equation \ref{eq2} yields,  $(r(S\setminus z)+1)+r(T)-r(M)-1 < 1,$ that is $r(S\setminus z)+r(T)-r(M) < 1 .$ Therefore $\{S\setminus z, T\}$ gives $1-$separation of $M,$ a contradiction. Further, if  $T$ contains an $np$-circuit, then $r'(S)=r(S\setminus z)+1,$ $r'(T)=r(T)+1.$ By Equation \ref{eq2}, we get $(r(S\setminus z)+1)+(r(T)+1)-r(M)-1 < 1 ,$ which gives $r(S\setminus z)+r(T)-r(M) < 0 ,$ which is not possible. So in either case such separation does not exist. Therefore $M'_{a,b}$ is connected. 
\end{proof}
\noindent In Example \ref{ex1}, the $p$-matroid $R_8 \cong M[A]$ and its element splitting $p$-matroid $M'_{3,5} \cong M[A'_{3,5}]$ both are connected. In the next result we give a necessary and sufficient condition to preserve $3$-connectedness of a $p$-matroid under the element splitting operation.
\begin{theorem}
	Let $M$ be a $3$-connected $p$-matroid. Then $M'_{a,b}$ is $3$-connected $p$-matroid if and only if for every $t\in E(M)$ there is an $np$-circuit of $M$ not containing $t.$ 
\end{theorem}
\begin{proof}
	Let $M'_{a,b}$ be $3$-connected $p$-matroid. On contrary, if there is an element $t\in E(M)$ contained in every $np$-circuit of $M.$ Take $S=\{z, t\}$ and $T=E\setminus S.$ Then $r'(S)+r'(T)-r'(M'_{a,b})=r(\{t\})+1+r(T)-r(M)-1=r(\{t\})+r(T)-r(M)=1<2.$ Because, in this case, $t\in cl(T)$ hence $r(T)=r(M).$ That is $\{S, T\}$ forms a $2$-separation of $M'_{a,b},$ a contradiction.\\
	For converse part suppose, for every $t\in E(M)$ there is an $np$-circuit of $M$ not containing $t.$ On the contrary assume that  $M'_{a,b}$ is not a $3$-connected matroid. Then there exists a $k$ separation, for $k \leq 2,$ of $M'_{a,b}.$ By Theorem \ref{con1}, $k$ can not be equal to $1$. For $k=2,$ let $\{S,T\}$ be a $2$-separation of $M'_{a,b}.$ Then $\{S,T\}$ is a partition of $E\cup \{z\}$ such that $|S|, |T|\geq 2$ and 
	\begin{equation}\label{eq3}
		r'(S) + r'(T)-r'(M'_{a,b}) <2.
	\end{equation}
	$\mathbf{Case~1:}$Suppose $S=\{z,t\},$ $t\in E(M).$ By hypothesis, $T$ contains an $np$-circuit not containing $t.$ Then Equation \ref{eq3} gives, $(r(\{t\})+1)+(1+r(T))-r(M)-1 < 2$ $\implies$ $r({t})+r(T)-r(M)< 1.$ Thus $\{\{t\}, T\}$ forms a $1$-separation of $M,$ which is a contradiction.\\
	$\mathbf{Case~2:}$ Suppose $ z\in S$ and $|S|\geq 3.$ If $T$ contains no $np$-circuit then Equation \ref{eq3} yields $(r(S\setminus z)+1)+r(T)-r(M)-1 < 2 \implies$ $r(S\setminus z)+r(T)-r(M) < 2.$ Therefore $\{S\setminus z, T\}$ gives a $2$-seperation of $M,$ a contradiction.\\
	Further, if  $T$ contains an $np$-circuit, then $r'(S)=r(S\setminus z)+1,$ $r'(T)=r(T)+1.$ By Equation \ref{eq3}, we get $(r(S\setminus z)+1)+(r(T)+1)-r(M)-1 < 2 \implies$  $r(S\setminus z)+r(T)-r(M) < 1.$ Thus, $\{S\setminus z, T\}$ gives a $1$-seperation of $M,$ a contradiction. So in either case such partition does not exist. Therefore $M'_{a,b}$ is $3$-connected.
\end{proof}
\section{Applications}
\noindent For Eulerian matroid $M$ on the ground set $E$ there exists disjoint circuits $C_1,$$C_2,$ $\ldots$,$C_k$ of $M$ such that $E= C_1 \cup C_2 \cup ...\cup C_k.$
\noindent We call the collection $\{C_1,C_2,\ldots,C_k\}$ a circuit decomposition of $M.$\\

\noindent Let $\{a,b\} \subset E$. We say a circuit decomposition \~{C} $=\{C_1,C_2,\ldots,C_k\}$ of $M$ an \textit{$ep$-decomposition} of $M$ if it contains exactly one $np$-circuit with respect to the $a,b$ splitting of $M.$
In the next proposition, we give a sufficient condition to yield  Eulerian $p$-matroids from Eulerian $p$-matroids after the element splitting operation.

\begin{proposition}
	Let $M$ be Eulerian $p$-matroid and $a,b\in E$. If $M$ has an \textit{$ep$-decomposition}, then $M'_{a,b}$ is Eulerian $p$-matroid.
\end{proposition}

\begin{proof}
	Let \~{C} $=\{C_1,C_2,\ldots,C_k\}$ be an \textit{$ep$-decomposition} of $M$ and $C_1$ be an $np$-circuit in it.  Then $C_1 \cup z$ is a circuit of  $M'_{a,b}.$ Thus $\{C_1 \cup z, C_2, \ldots ,C_k\}$ is the desired circuit decomposition of $M'_{a,b}.$ 
	
\end{proof}

\begin{proposition}
	Let $M'_{a,b}$ is Eulerian $p$-matroid and \~{C}$=\{C_1, C_2, \ldots ,C_k\}$ be a circuit decomposition of $M'_{a,b}$. If \~{C} contains no member which is a union of an $np$-circuit and an independent set of $M,$ then $M$ is Eulerian and has an $ep$-decomposition.
	
\end{proposition}
\begin{proof}
	Assume, without loss of generality, $z \in C_1.$ Then $C_1 \in \mathcal{C}_z$ and $C_1 \setminus z$ is an $np$-circuit of $M.$ We will show $C_1 \setminus z$ contains both $a$ and $b.$ On the contrary assume that $C_1 \setminus z$ contains only $a.$ Then $b \in C_i$ for some $i \in \{2,3,\ldots,k\}.$ Since $C_i$ is also a circuit of $M_{a,b}$ containing only $b,$ by Theorem 2.10 of \cite{mjg} it must be a union of an $np$-circuit and an independent set of $M,$ which is a contradiction to the hypothesis. Therefore $C_1 \setminus z$ contains both  $a$ and $b$ and the collection $\{C_1\setminus z, C_2, \ldots ,C_k\}$ forms an $ep$-decomposition of $M.$
	
\end{proof}

\noindent In Example \ref{ex1}, the matroid $R_8$ is Eulerian with $ep$-decomposition $E=C_1 \cup C_2,$ where $C_1 =\{2,4,6,8\}$ is a $p$-circuit and $C_2=\{1,3,5,7\}$ is an $np$-circuit. An element splitting matroid $M'_{3,5}$ is also Eulerian with circuit decomposition $E \cup z=C_1 \cup (C_2 \cup z).$

\noindent M. Borowiecki \cite{bro} defined hamiltonian matroid as a matroid containing a circuit of size $r(M)+1.$ This circuit is called the hamiltonian circuit of the matroid $M.$ In the next corollary, we give a sufficient condition to yield hamiltonian matroid from hamiltonian matroid after the element splitting operation.

\begin{corollary}
If $M$ is hamiltonian matroid with an $np$-circuit of size $r(M)+1,$ then $M'_{a,b}$ is hamiltonian.
\end{corollary} 
\begin{proof}
Let $C$ be an $np$-circuit of $M$ of size $r(M)+1.$ Then by Proposition \ref{L1}, $C\cup z$ is a circuit in $M'_{a,b}$ of size $r(M)+2.$
\end{proof}
\noindent In Example \ref{ex1}, the matroid $R_8 \cong M[A]$ is hamiltonian and its element splitting matroid $M'_{3,5} \cong M[A'_{3,5}]$ is also hamiltonian.\\

\noindent Rota conjectured that the family of matroids that are representable over finite fields has only finitely many excluded minors \cite{GGW1}. For example, the 4-point line, $U_{2,4},$ is the only excluded minor for the class of binary matroids. In the following example, we demonstrate that there exist a splitting of the ternary matroid $U_{2,4},$ which yields a graphic matroid.

\begin{example}
	Let the matrix $A$ represents the ternary matroid $U_{2,4}$ and the vector matroid of $A_{1, 3}$ represents the splitting matroid $M[A_{1,3}]$.\\
	$\mathbf{A} = 
	\begin{pNiceMatrix}%
		[first-col,
		first-row,
		code-for-first-col = \color{black},
		code-for-first-row = \color{black}]
		& 1 & 2 & 3 & 4  \\
		& 1 & 0 & 1 & 1  \\
		& 0 & 1 & 1 & 2  \\
	\end{pNiceMatrix}\qquad
	\mathbf{A_{1,3}} = 
	\begin{pNiceMatrix}%
		[first-col,
		first-row,
		code-for-first-col = \color{black},
		code-for-first-row = \color{black}]
		& 1 & 2 & 3 & 4  \\
		& 1 & 0 & 1 & 1  \\
		& 0 & 1 & 1 & 2  \\
		& 1 & 0 & 1 & 0
	\end{pNiceMatrix} \qquad
\mathbf{A'_{1,3}} = 
\begin{pNiceMatrix}%
[first-col,
first-row,
code-for-first-col = \color{black},
code-for-first-row = \color{black}]
& 1 & 2 & 3 & 4 & 5  \\
& 1 & 0 & 1 & 1 & 0 \\
& 0 & 1 & 1 & 2 & 0 \\
& 1 & 0 & 1 & 0 & 1
\end{pNiceMatrix}$. \\
Observe that
\begin{itemize}
	\item the splitting matroid $M[A_{1,3}]$ is binary and matrix $B  = \begin{pNiceMatrix}%
		
		1 & 0 & 1 & 1  \\
		0 & 1 & 1 & 0  \\
		0 & 0 & 0 & 1
	\end{pNiceMatrix}$  
	gives its binary representation. 
	\item $A'_{1,3} /5 = U_{2,4}.$
	
\end{itemize} 

\end{example}

\noindent  However, the element splitting operation on $U_{2, 4}$ does not give a binary matroid. With this observation, we propose the following question:

 For a given ternary matroid $M,$ does there always exist a pair of elements $\{a, b\}$ in $E(M)$ such that the splitting matroid $M_{a, b}$ is binary (graphic)?

\noindent \textbf{Funding Details:}\\
The Authors received no financial support for this work.\\
\noindent \textbf{Conflict of Interest:}\\
The authors report there are no competing interests to declare.

\end{document}